\documentclass[12pt]{article}
\usepackage{amsmath, amsthm, amssymb}

\title{Problems of
classifying associative or Lie algebras
and triples of symmetric or
skew-symmetric matrices are wild\footnotetext{This is the authors' version of a work that was published in Linear Algebra Appl. 407 (2005) 249--262.}}
\author{Genrich Belitskii%
\thanks{Partially supported by
Israel Science Foundation, Grant
186/01.}\\ Dept. of Mathematics,
Ben-Gurion University of the Negev,\\
Beer-Sheva 84105, Israel,
 genrich@cs.bgu.ac.il
 \and
Ruvim Lipyanski\\ Dept. of
Mathematics,
Ben-Gurion University of the Negev,\\
Beer-Sheva 84105, Israel,
 lipyansk@cs.bgu.ac.il
 \and
Vladimir V. Sergeichuk%
\thanks{Corresponding author.
The research was started while this
author was visiting the Ben-Gurion
University of the Negev.}\\
Institute of Mathematics,
Tereshchenkivska 3, Kiev,
Ukraine,\\sergeich@imath.kiev.ua}
\date{}

\DeclareMathOperator{\rank}{rank}
\DeclareMathOperator{\Rad}{Rad}
\DeclareMathOperator{\diag}{diag}

\renewcommand{\ge}{\geqslant}

\newtheorem{theorem}{Theorem}
\newtheorem{lemma}[theorem]{Lemma}
\newtheorem{corollary}[theorem]{Corollary}

\theoremstyle{definition}

\theoremstyle{remark}
\newtheorem{remark}[theorem]{Remark}

\begin{document}
\maketitle

\begin{abstract}
We prove that the problems of
classifying triples of symmetric or
skew-symmetric matrices up to
congruence, local commutative
associative algebras with zero cube
radical and square radical of dimension
$3$, and Lie algebras with central
commutator subalgebra of dimension $3$
are hopeless since each of them reduces
to the problem of classifying pairs of
$n$-by-$n$ matrices up to simultaneous
similarity.

{\it AMS classification:} 17B30;
15A21; 16G60

{\it Keywords:} Nilpotent Lie
algebras; Bilinear forms; Wild
problems
 \end{abstract}



\section{Introduction}
\label{s0}

All matrices, vector spaces,
and algebras are considered
over an algebraically closed
field $\mathbb F$ of
characteristic other than
two.

The problem of classifying pairs
of $n\times n$ matrices up to
similarity transformations
\begin{equation*}\label{0a}
(A,B)\longmapsto S^{-1}(A,B)S:=
(S^{-1}AS,\,S^{-1}BS),
\end{equation*}
in which $S$ is any
nonsingular $n\times n$
matrix, is hopeless since it
contains the problem of
classifying an arbitrary
system of linear operators
and the problem of
classifying representations
of an arbitrary
finite-dimensional algebra,
see \cite{bel-ser}.
Classification problems that
contain the problem of
classifying pairs of matrices
up to similarity are called
\emph{wild}.

We prove the wildness of the
problems of classifying
\begin{itemize}
  \item[(i)]
triples of Hermitian forms (with
respect to a nonidentity involution
on $\mathbb F$),

  \item[(ii)]
for each $\varepsilon_1,
\varepsilon_2,
\varepsilon_3\in\{1,-1\}$, triples
of bilinear forms $({\cal A}_1,{\cal
A}_2,{\cal A}_3)$, in which ${\cal
A}_i$ is symmetric if
$\varepsilon_i=1$ and skew-symmetric
if $\varepsilon_i=-1$,

\item[(iii)]
local commutative associative
algebras $\Lambda$ over $\mathbb F$
with $(\Rad\Lambda)^3=0$ and $\dim
(\Rad\Lambda)^2=3$, and

  \item[(iv)]
Lie algebras $L$ over $\mathbb F$
with central commutator subalgebra
of dimension $3$.
\end{itemize}
The hopelessness of the
problems of classifying
triples (i) and (ii) was also
proved in \cite{ser2} by
another method (which was
used in \cite{ser1} too):
each of them reduces to the
problem of classifying
representations of a wild
quiver. The wildness of the
problem of classifying local
associative algebras
$\Lambda$ with
$(\Rad\Lambda)^3=0$ and $\dim
(\Rad\Lambda)^2=2$ was proved
in \cite{bbp}.

Recall that an \emph{algebra}
$\Lambda$ over $\mathbb F$ is a
finite dimensional vector space
being also a ring such that
\begin{equation*}\label{eqy}
\alpha (ab)=(\alpha a)b=a(\alpha
b)
\end{equation*}
for all $\alpha\in\mathbb F$ and
all $a,b\in\Lambda$. An algebra
$\Lambda$ is \emph{local} if
there exists an ideal $R$ such
that $\Lambda/R$ is isomorphic
to $\mathbb F$ (then $R$ is the
\emph{radical} of $\Lambda$ and
is denoted by $\Rad\Lambda$).

A \emph{Lie algebra $L$ with
central commutator subalgebra}
is a vector space with
multiplication given by a
skew-symmetric bilinear mapping
\[
[\ ,\ ]\colon  L\times L
\longrightarrow L
\]
that satisfies $[[a,b],c]=0$ for
all $a,b,c\in L$. The
\emph{commutator subalgebra}
$L^2$ is the subspace spanned by
all $[a,b]$.

\section{Triples of forms}
\label{sec1}

Let $a\mapsto \bar{a}$ be any
involution on $\mathbb F$, that is,
a bijection $\mathbb F\to\mathbb F$
such that
\begin{equation*}\label{1ig}
 \overline{a+b}= \bar{a}+\bar{b},
 \quad \overline{ab}=\bar{a}
 \bar{b},\quad
 \bar{\bar{a}}=a.
\end{equation*}
For a matrix $A=[a_{ij}]$, we define
\[ A^*:=\bar{A}^T =[\bar{a}_{ji}].\]
If $S^*AS=B$ for a nonsingular
matrix $S$, then $A$ and $B$ are
said to be *{\it\!congruent}. The
involution $a\mapsto \bar{a}$ can be
the identity; we consider congruence
of matrices as a special case of
*congruence.

Each matrix tuple in this paper is
formed by matrices of the same size,
which is called the size of the
tuple. Denote
\[
R(A_1,\dots,A_t):=
(RA_1,\dots,RA_t), \qquad
(A_1,\dots,A_t)S:=
(A_1S,\dots,A_tS).
\]
We say that matrix tuples
$(A_1,\dots,A_t)$ and
$(B_1,\dots,B_t)$ are
\emph{equivalent} and write
\begin{equation}\label{er}
(A_1,\dots,A_t)\sim (B_1,\dots,B_t)
\end{equation}
if there exist nonsingular $R$ and
$S$ such that
\[
R(A_1,\dots,A_t)S=(B_1,\dots,B_t).
\]
These tuples are *\!\emph{congruent}
if $R=S^*$.

Denote by $I_n$ the $n\times n$
identity matrix, by $0_{mn}$ the
$m\times n$ zero matrix, and
abbreviate $0_{nn}$ to $0_{n}$.

For $\varepsilon_1, \varepsilon_2,
\varepsilon_3\in\mathbb F$, define
the triple
\begin{equation}\label{eqv1}
{\cal T}_{\varepsilon}(x,y):= \left(
\begin{bmatrix}
  0_4&I_4\\
  \varepsilon_1 I_4&0_4
\end{bmatrix},
        \
\begin{bmatrix}
  0_4&J_4(0)\\
  \varepsilon_2 J_4(0)^T&0_4
\end{bmatrix},
         \
\begin{bmatrix}
  0_4&D(x,y)\\
  \varepsilon_3 D(x^*,y^*)&0_4
\end{bmatrix}\right)
\end{equation}
of polynomial matrices in $x,\ y,\
x^*$, and $y^*$, in which
\begin{equation}\label{qwq1}
J_4(0):= \begin{bmatrix}
0&1&0&0\\0&0&1&0
\\0&0&0&1\\0&0&0&0
  \end{bmatrix},
  \qquad
D(x,y):= \begin{bmatrix}
1&0&0&0\\0&x&0&0\\0&0&y&0\\0&0&0&0
  \end{bmatrix}.
\end{equation}
For each pair $(A,B)$ of $n$-by-$n$
matrices, define ${\cal
T}_{\varepsilon}(A,B)=$
\begin{equation}\label{eqv2}
\left(
\begin{bmatrix}
  0_{4n}&I_{4n}\\
  \varepsilon_1 I_{4n}&0_{4n}
\end{bmatrix},
        \
\begin{bmatrix}
  0_{4n}&J_{4}(0_n)\\
  \varepsilon_2 J_{4}(0_n)^T&0_{4n}
\end{bmatrix},
         \
\begin{bmatrix}
  0_{4n}&D(A,B)\\
  \varepsilon_3 D(A^*,B^*)&0_{4n}
\end{bmatrix}\right),
\end{equation}
where
\begin{equation}\label{qwq}
J_{4}(0_n)=\begin{bmatrix}
0_n&I_n&0&0\\0&0_n&I_n&0
\\0&0&0_n&I_n\\0&0&0&0_n
  \end{bmatrix},
  \qquad
D(A,B)= \begin{bmatrix}
I_n&0&0&0\\0&A&0&0\\0&0&B&0\\0&0&0&0_n
  \end{bmatrix}.
\end{equation}

We prove in this section the
following theorem; its statement (a)
is used in the next section.

\begin{theorem}\label{th2.1}
Let $\mathbb F$ be an
algebraically closed field of
characteristic other than
two.

{\rm(a)} For nonzero $\varepsilon_1,
\varepsilon_2\in\mathbb F$ and any
$\varepsilon_3\in\mathbb F$, matrix
pairs $(A,B)$ and $(C,D)$ over
$\mathbb F$ are similar if and only
if ${\cal T}_{\varepsilon}(A,B)$ and
${\cal T}_{\varepsilon}(C,D)$ are
{\rm *}\!congruent.

{\rm(b)} The problems of classifying
triples {\rm(i)} and {\rm(ii)} from
Section \ref{s0} are wild.
\end{theorem}

Define the \emph{direct sum} of
matrix tuples:
\[
(A_1,\dots,A_t)\oplus(B_1,\dots,B_t)
:=(A_1\oplus B_1,\dots,A_t\oplus
B_t).
\]
We say that a tuple ${\cal T}_1$ of
$p\times q$ matrices is a
\emph{direct summand of a tuple
$\cal T$ for equivalence} if $p+q>
0$ and ${\cal T}$ is equivalent to
${\cal T}_1\oplus{\cal T}_2$ for
some ${\cal T}_2$. If also $p=q$ and
${\cal T}$ is *congruent to ${\cal
T}_1\oplus{\cal T}_2$, then ${\cal
T}_1$ is a \emph{direct summand of
$\cal T$ for {\rm *}\!congruence}. A
matrix tuple is
\emph{indecomposable} with respect
to equivalence (*congruence) if it
has no direct summand of a smaller
size for equivalence (*congruence).

\begin{lemma} \label{lem2}
{\rm(a)} Each tuple of $m$-by-$n$
matrices is equivalent to a direct
sum of tuples that are
indecomposable with respect to
equivalence. This sum is determined
uniquely up to permutation of
summands and replacement of summands
by equivalent tuples.

{\rm(b)} Each tuple of $n$-by-$n$
matrices is {\rm*}\!congruent to a
direct sum of tuples that are
indecomposable with respect to
{\rm*}\!congruence. This sum is
determined uniquely up to permutation
of summands and replacement of summands
by {\rm*}\!congruent tuples.
\end{lemma}

\begin{proof}
(a) Each $t$-tuple of $m\times n$
matrices determines the $t$-tuple of
linear mappings ${\mathbb F}^n\to
{\mathbb F}^m$; that is, the
representation of the quiver
consisting of two vertices $1$ and
$2$ and $t$ arrows $1\longrightarrow
2$. By the Krull--Schmidt theorem
\cite[Section 8.2]{pie}, every
representation of a quiver is
isomorphic to a direct sum of
indecomposable representations
determined up to replacement by
isomorphic representations and
permutations of summands.

(b) This statement is a special case
of the following generalization of
the law of inertia for quadratic
forms \cite[Theorem 2 and
\S\,2]{ser1}: each system of linear
mappings and sesquilinear forms on
vector spaces over $\mathbb F$
decomposes into a direct sum of
indecomposable systems uniquely up
to isomorphisms of summands.
\end{proof}

It is worthy of note that the
uniqueness of decompositions in
Lemma \ref{lem2}(a) holds only if we
suppose that there exists exactly
one matrix of size $0\times n$ and
there exists exactly one matrix of
size $n\times 0$ for every
nonnegative integer $n$; they give
the linear mappings ${\mathbb
F}^n\to 0$ and $0\to {\mathbb F}^n$
and are considered as zero matrices
$0_{0n}$ and $0_{n0}$. Then for any
$m$-by-$n$ matrix $M$
\begin{equation}\label{4w}
M\oplus 0_{p0}=\begin{bmatrix} M\\
 0_{pn}
\end{bmatrix}\qquad\text{and}\qquad
M\oplus 0_{0q}=\begin{bmatrix} M&
 0_{mq}
\end{bmatrix}.
\end{equation}
In particular, $0_{p0}\oplus
0_{0q}=0_{pq}$.

\begin{lemma} \label{lemm}
{\rm(a)} Every direct summand for
equivalence of
\begin{equation}\label{mjk}
{\cal G}= \left( I_{4n},\ J_4(0_n),\
D\right)
\end{equation}
$($in which $J_4(0_n)$ is defined in
\eqref{qwq} and $D$ is any
$4n$-by-$4n$ matrix$)$ reduces by
equivalence transformations to the
form
\begin{equation}\label{fg}
{\cal G}'=(I_{4p},\ J_4(0_p),\ M').
\end{equation}

{\rm(b)} If \eqref{fg} is a direct
summand for equivalence of the tuple
\eqref{mjk} with
\begin{equation}\label{ftf}
D=\diag(\alpha I_n,\:A,\:B,\:\beta I_n)
\end{equation}
$(A$ and $B$ are $n$-by-$n$ and
$\alpha,\beta\in\mathbb F)$, then
$M'$ has the form
\begin{equation}\label{fgg}
M'=\begin{bmatrix} \alpha
I_p&M'_{12}&M'_{13}&M'_{14}
\\0&M'_{22}&M'_{23}&M'_{24}
\\0&0&M'_{33}&M'_{34}\\0&0&0&\beta I_p
  \end{bmatrix}
\end{equation}
$($all blocks are $p$-by-$p)$.
\end{lemma}

\begin{proof}
(a) Let ${\cal G}'$ be a direct
summand for equivalence of the tuple
\eqref{mjk}; this means that $ {\cal
G}\sim{\cal G}'\oplus{\cal G}'' $
(in the notation \eqref{er}) for
some ${\cal G}''$. The first matrix
of the triple ${\cal G}$ is the
identity, so we can reduce the first
matrix of ${\cal G}'\oplus{\cal G}''
$ to the identity matrix too by
equivalence transformations with
${\cal G}'$ and ${\cal G}''$. Then
the equivalence of ${\cal G}$ and
${\cal G}'\oplus{\cal G}''$ means
that their second matrices are
similar. The second matrix of ${\cal
G}$ is similar to the Jordan matrix
$J_4(0)\oplus \dots\oplus J_4(0)$,
and so we can reduce the second
matrix of ${\cal G}'$ to $J_4(0_p)$
by equivalence transformations with
${\cal G}'$. This proves (a).

(b) Let \eqref{fg} be a direct
summand for equivalence of the tuple
\eqref{mjk} with $D$ of the form
\eqref{ftf}. Then ${\cal G}\sim{\cal
G}'\oplus{\cal G}''$ for some ${\cal
G}''$. By (a), ${\cal G}''$ can be
taken in the form ${\cal
G}''=(I_{4q},\ J_4(0_q),\ M'')$ with
$q:=n-p$. Partition $M'$ and $M''$
into $p$-by-$p$ and $q$-by-$q$
blocks:
\[
M'=[M'_{ij}]_{i,j=1}^4,\qquad
M''=[M''_{ij}]_{i,j=1}^4.
\]
Using simultaneous permutations of rows
and columns of the matrices of ${\cal
G}'\oplus{\cal G}''$, we construct the
equivalence
\begin{equation}\label{cac}
{\cal G}'\oplus{\cal G}''\sim (I_{4n},\
J_4(0_n),\ M),\qquad M:=[M'_{ij}\oplus
M''_{ij}]_{i,j=1}^4.
\end{equation}
Since  ${\cal G}\sim{\cal
G}'\oplus{\cal G}''$, we have ${\cal
G}\sim (I_{4n},\ J_4(0_n),\ M)$, and so
there exist nonsingular $R$ and $S$
such that
\begin{equation}\label{bnm}
{\cal G}S= R(I_{4n},\ J_4(0_n),\ M).
\end{equation}
Equating the corresponding matrices of
the triples \eqref{bnm} gives
\[
I_{4n}S=RI_{4n},\qquad
J_{4}(0_n)S=RJ_{4}(0_n),\qquad DS=RM.
\]
By the first and the second equalities,
\begin{equation}\label{nhj}
 S=R=\begin{bmatrix}
S_0&S_1&S_2&S_3\\0&S_0&S_1&S_2
\\0&0&S_0&S_1\\0&0&0&S_0
\end{bmatrix}.
\end{equation}
By the third equality and \eqref{ftf},
$M$ has the form
\begin{equation*}\label{mkj}
\begin{bmatrix}
\alpha I_n&M_{12}&M_{13}&M_{14}
\\0&M_{22}&M_{23}&M_{24}
\\0&0&M_{33}&M_{34}\\0&0&0&\beta I_n
  \end{bmatrix}.
\end{equation*}
Since $M$ is defined by \eqref{cac},
$M'$ has the form \eqref{fgg}.
\end{proof}

\begin{proof}[Proof of Theorem
\ref{th2.1}] (a) If $(A,B)$ is
similar to $(C,D)$, then ${\cal
T}_{\varepsilon}(A,B)$ is *congruent
to ${\cal T}_{\varepsilon}(C,D)$
since $S^{-1}(A,B)S=(C,D)$ implies
\begin{gather*}\label{equ10}
R^*{\cal
T}_{\varepsilon}(A,B)R={\cal
T}_{\varepsilon}(C,D),
\\ \label{iii}
R:=\diag ((S^*)^{-1},(S^*)^{-1},
(S^*)^{-1},(S^*)^{-1},S,\,S,\,S,\,S).
\end{gather*}

Conversely, suppose that ${\cal
T}_{\varepsilon}(A,B)$ is *congruent
to ${\cal T}_{\varepsilon}(C,D)$.
Then they are equivalent, and so
\begin{equation}\label{erq}
{\cal G}(A,B)\oplus {\cal
H}_{\varepsilon}(A,B)\sim {\cal
G}(C,D)\oplus {\cal
H}_{\varepsilon}(C,D),
\end{equation}
where
\begin{align} \label{kk2}
{\cal G}(X,Y)&:= ( I_{4n},\:
J_4(0_n),\: D(X,Y)),
    \\ \nonumber
{\cal H}_{\varepsilon}(X,Y)&:= (
\varepsilon_1 I_{4n},\:
\varepsilon_2 J_4(0_n)^T,\:
\varepsilon_3 D(X^*,Y^*))
\end{align}
for $n$-by-$n$ matrices $X$
and $Y$. Let
$\mu_2:=\varepsilon_2/
\varepsilon_1$ and
$\mu_3:=\varepsilon_3/
\varepsilon_1$, then
\[
{\cal H}_{\varepsilon}(X,Y)\sim{\cal
H}_{\mu}(X,Y)= (I_{4n},\: \mu_2
J_4(0_n)^T,\: \mu_3 D(X^*,Y^*)).
\]
Furthermore, let
\[S:=\diag(I_n,\:\mu_2 I_n,\: \mu_2^2
I_n,\: \mu_2^3
I_n),\qquad\nu:=\mu_3,\] then
\[
{\cal H}_{\mu}(X,Y)\sim S^{-1}{\cal
H}_{\mu}(X,Y)S= (I_{4n},\:
J_4(0_n)^T,\: \nu D(X^*,Y^*))= {\cal
H}_{\nu}(X,Y).
\]
Lastly,
\[
{\cal H}_{\nu}(X,Y)\sim P{\cal
H}_{\nu}(X,Y)P= ( I_{4n},\: J_4(0_n),\:
\nu D'(X^*,Y^*)),
\]
where
\[
P=\begin{bmatrix} 0&0&0&I_n
\\0&0&I_n&0
\\0&I_n&0&0\\I_n&0&0&0
  \end{bmatrix},\qquad
 D'(X^*,Y^*)= \begin{bmatrix}
0_n&0&0&0\\0&Y^*&0&0
\\0&0&X^*&0\\0&0&0&I_n
  \end{bmatrix}.
\]
Therefore,
\[
{\cal H}_{\varepsilon}(X,Y)\sim{\cal
H}'(X,Y):=( I_{4n},\: J_4(0_n),\:
\nu D'(X^*,Y^*)),
\]
and by \eqref{erq}
\begin{equation}\label{erq1}
{\cal G}(A,B)\oplus {\cal
H}'(A,B)\sim {\cal G}(C,D)\oplus
{\cal H}'(C,D).
\end{equation}

Suppose that ${\cal G}(A,B)$ and
${\cal H}'(C,D)$ have a common
direct summand ${\cal G}'$ for
equivalence. By Lemma \ref{lemm}(a),
we may take ${\cal G}'=(I_{4p},\
J_4(0_p),\ M')$. Moreover, since
$D(A,B)$ and $\nu D'(C^*,D^*)$ are
of the form \eqref{ftf} with
$\alpha=1$ and $\alpha=0$,
respectively, by Lemma \ref{lemm}(b)
the matrix $M'$ has the form
\eqref{fgg} with $\alpha=1$ and
$\alpha=0$ simultaneously, a
contradiction.

Hence ${\cal G}(A,B)$ and ${\cal
H}'(C,D)$ have no common direct
summands for equivalence. The
triples ${\cal G}(C,D)$ and ${\cal
H}'(A,B)$ have no common direct
summands too. By \eqref{erq1} and
Lemma \ref{lem2}(a), ${\cal
G}(A,B)\sim{\cal G}(C,D)$; that is,
${\cal G}(A,B)S=R{\cal G}(C,D)$ for
some nonsingular $R$ and $S$.
Equating the corresponding matrices
of these triples gives \eqref{nhj}
and $(A,B)S_0=S_0(C,D)$; that is,
$(A,B)$ is similar to $(C,D)$.
\medskip

(b) If the involution on $\mathbb F$
is not the identity and
$\varepsilon_1=
\varepsilon_2=\varepsilon_3 =1$,
then the matrices of the triple
\eqref{eqv2} are Hermitian. If the
involution on $\mathbb F$ is the
identity and $\varepsilon_1,
\varepsilon_2,\varepsilon_3\in
\{1,-1\}$, then each matrix of the
triple \eqref{eqv2} is symmetric or
skew-symmetric. So the statement (b)
of Theorem \ref{th2.1} follows from
the statement (a).
\end{proof}

\section{Algebras}
\label{s1}

We consider (associative) algebras
and Lie algebras as special cases of
semialgebras. By a
\emph{semialgebra} we mean a
finite-dimensional vector space $R$
over $\mathbb F$ with multiplication
given by a bilinear mapping
$(a,b)\mapsto ab\in R$:
\[ (\alpha a+\beta b) c
=\alpha  (ac) +\beta (bc),
 \qquad
a(\alpha b+\beta c) =\alpha
(ab)+\beta (ac)\] for all $\alpha
,\beta \in\mathbb F$ and all $a,b,c
\in R$. A semialgebra $R$ is
\emph{commutative} or
\emph{anti-commutative} if $ab=ba$
or, respectively, $ab=-ba$ for all
$a,b\in R$. Denote by $R^2$ and
$R^3$ the vector spaces spanned by
all $ab$ and, respectively, by all
$(ab)c$ and $a(bc)$, where $a,b,c\in
R$.

An \emph{algebra} $\Lambda$ over
$\mathbb F$ is an associative
semialgebra with the identity 1:
\[
(ab)c=a(bc),\qquad 1a=a\qquad(a,b,c\in
\Lambda).
\]
An algebra $\Lambda$ is \emph{local}
if the set $R$ of its noninvertible
elements is closed under addition.
Then $R$ is the \emph{radical} and
$\Lambda/R$ is isomorphic to
$\mathbb F$ (see \cite[Section
5.2]{pie}).

A \emph{Lie algebra} $L$ over
$\mathbb F$ is an anti-commutative
semialgebra whose multiplication is
denoted by $[\ ,\ ]$ and satisfies
the Jacobi identity
\begin{equation}
\label{1.1b}
[[a,b],c]+[[b,c],a]+[[c,a],b]=0
\end{equation}
for all $a,b,c\in L$. Then $L^2$ is
called the \emph{commutator
subalgebra} of $L$. The commutator
subalgebra is \emph{central} if
$L^3=0$, that is, if
\begin{equation*}\label{1.v}
\text{$[[a,b],c]=0$\quad for all
$a,b,c\in L$;}
\end{equation*}
the last equality implies
\eqref{1.1b}. A Lie algebra with
central commutator subalgebra is
also called a \emph{two-step
nilpotent Lie algebra}. Due to the
next theorem, the full
classification of such Lie algebras
is impossible; one can consider its
special cases or reduce it to
another classification problem of
the same complexity; see, for
instance, \cite[Theorems 2 and
3]{mas}.

\begin{theorem} \label{theor}
Let $\mathbb F$ be an
algebraically closed field of
characteristic other than
two.

{\rm(a)} The problem of classifying
local commutative algebras $\Lambda$
over $\mathbb F$ with
$(\Rad\Lambda)^3=0$ and $\dim
(\Rad\Lambda)^2=3$ is wild.

{\rm(b)} The problem of classifying
Lie algebras over $\mathbb F$ with
central commutator subalgebra of
dimension $3$ is wild.
\end{theorem}

By the next lemma, the problems
considered in Theorem \ref{theor} are
the problems of classifying commutative
or anti-commutative semialgebras $R$
with $R^3=0$ and $\dim R^2=3$ (these
semialgebras are associative and
satisfy \eqref{1.1b} due to $R^3=0$).

\begin{lemma} \label{l5za}
Let $R$ be a semialgebra with
$R^3=0$ and $\dim R^2=3$.

{\rm(a)} $R$ is commutative if and
only if $R$ is the radical of some
algebra $\Lambda$ from Theorem
{\rm\ref{theor}(a)}; moreover,
$\Lambda$ is fully determined by
$R$.

{\rm(b)} $R$ is anti-commutative if
and only if $R$ is a Lie algebra
from Theorem {\rm\ref{theor}(b)}.
\end{lemma}

\begin{proof}
Let $R$ be a semialgebra with $R^3=0$
and $\dim R^2=3$.

(a) If $R$ is commutative, then we
``adjoin'' the identity $1$ by
considering the algebra $\Lambda$
consisting of the formal sums
\[
\alpha 1+a\qquad (\alpha\in\mathbb F,\
a\in R)
\]
with the componentwise addition and
scalar multiplication and the
multiplication
\[
(\alpha 1+a)(\beta 1+b)= \alpha\beta
1+(\alpha b+\beta a+ab).
\]
This multiplication is associative
since $R^3=0$, and so $\Lambda$ is a
commutative algebra. Since $R$ is
the set of its noninvertible
elements, $\Lambda$ is a local
algebra and $R$ is its radical.

(b) If $R$ is anti-commutative, then
$R$ is a Lie algebra since \eqref{1.1b}
holds due to $R^3=0$.
\end{proof}

\begin{lemma} \label{l5z}
Every semialgebra $R$ with $R^3=0$ and
$\dim R^2=t$ is isomorphic to exactly
one semialgebra on ${\mathbb F}^{n}$
with multiplication
\begin{equation}\label{5.1z}
uv=\left(u^T\!\begin{bmatrix}
0_t&0\\0&A_1
\end{bmatrix}\!v,\dots,
u^T\!\begin{bmatrix} 0_t&0\\0&A_t
\end{bmatrix}\!v,\,0,\dots,0\right)^T\!\!,
\end{equation}
given by a tuple $(A_1,\dots,A_t)$
of $(n-t)$-by-$(n-t)$ matrices that
are linearly independent; this means
that for all
$\alpha_1,\dots,\alpha_t\in\mathbb
F$
\begin{equation*}\label{5.2z}
\alpha_1A_1+\dots+\alpha_tA_t=0
\qquad\Longrightarrow\qquad
\alpha_1=\dots=\alpha_t=0.
\end{equation*}
The tuple $(A_1,\dots,A_t)$ is
determined by $R$ uniquely up to
congruence and linear substitutions
\begin{equation}\label{1.3z}
(A_1,\dots,A_t)\longmapsto (\gamma
_{11} A_1+\dots+\gamma _{1t}
A_t,\,\dots,\,\gamma _{t1}
A_1+\dots+\gamma _{tt} A_t),
\end{equation}
in which the matrix $[\gamma _{ij}]$
must be nonsingular. The semialgebra
$R$ is commutative or
anti-commutative if and only if all
the matrices $A_1,\dots,A_t$ are
symmetric or, respectively,
skew-symmetric.
\end{lemma}

\begin{proof}
Let $R$ be a semialgebra of dimension
$n$ with $R^3=0$ and $\dim R^2=t$.
Choose a basis $e_1,\dots,e_t$ of $R^2$
and complete it to a basis
\begin{equation}\label{bbb}
 e_1,\dots,e_t,f_1,\dots,f_{n-t}
\end{equation}
of $R$. Since $R^3=0$,
\begin{equation}\label{5.vz}
e_ie_j=0,\quad e_if_j=0,\quad f_if_j
=\alpha_{1ij}e_1+\dots
+\alpha_{tij}e_t,
\end{equation}
and the $(n-t)$-by-$(n-t)$ matrices
$A_1=[\alpha_{1ij}],\dots,
A_t=[\alpha_{tij}]$ are symmetric or
skew-symmetric if $R$ is commutative
or, respectively, anti-commutative.
Representing the elements of $R$ by
their coordinate vectors with
respect to the basis \eqref{bbb} and
using \eqref{5.vz}, we obtain
\eqref{5.1z}.  A change of the basis
$e_1,\dots,e_t$ of $R^2$ reduces
$(A_1,\dots,A_t)$ by transformations
\eqref{1.3z}. A change of the basis
vectors $f_1,\dots,f_{n-t}$ reduces
$(A_1,\dots,A_t)$ by congruence
transformations. The linear
independence of the system of
matrices $A_1,\dots,A_t$ follows
from \eqref{5.vz} because $\dim R^2=
t$.
\end{proof}

Due to Lemma \ref{l5z} and the next
lemma, the problem of classifying
commutative $($respectively,
anti-commutative$)$ semialgebras $R$
with $R^3=0$ and $\dim R^2=3$ is
wild. By Lemma \ref{l5za},
\emph{this proves Theorem}
\ref{theor}.

\begin{lemma} \label{t4.2}
The problem of classifying triples of
symmetric $($respectively,
skew-symmetric$)$ matrices up to
congruence and substitutions
\eqref{1.3z} with $t=3$ is wild.
\end{lemma}

\begin{proof}
Let $\varepsilon=1$
(respectively,
$\varepsilon=-1$), denote
\begin{equation*}\label{4x}
A^{\triangledown}
=\begin{bmatrix}0&A\\
\varepsilon A^T &0
\end{bmatrix}
\end{equation*}
for each matrix $A$, and denote
\begin{equation}\label{4x'q}
(A,\dots,D)^{\triangledown}=
(A^{\triangledown},\dots,
D^{\triangledown})
\end{equation}
for each matrix tuple $(A,\dots,D)$.

Consider the triple of
$350$-by-$350$ matrices
\begin{multline}\label{5.9'}
{\cal T}(x,y) :=(I_{100},0_{100},
0_{100})^{\triangledown}
  \oplus
(0_{50},I_{50}, 0_{50})^{\triangledown}
       \\
\oplus (0_{20},0_{20},
I_{20})^{\triangledown}\oplus
(I_1,I_1, I_1)^{\triangledown}\oplus
{\cal G}(x,y)^{\triangledown},
\end{multline}
in which
\begin{equation}\label{5.q'}
{\cal G}(x,y)= ( I_{4},\: J_4(0),\:
D(x,y))
\end{equation}
(see \eqref{qwq1} and \eqref{kk2}).

Let $(A,B)$ and $(C,D)$ be two pairs
of $n$-by-$n$ matrices. If $(A,B)$
is similar to $(C,D)$; that is,
$S^{-1}(A,B)S=(C,D)$ for some
nonsingular $S$, then ${\cal
G}(A,B)^{\triangledown}$ is
congruent to ${\cal
G}(C,D)^{\triangledown}$ since
\[
R^T{\cal
G}(A,B)^{\triangledown}R={\cal
G}(C,D)^{\triangledown},
\]
where
 \[
R:=(S^T)^{-1}\oplus (S^T)^{-1}\oplus
(S^T)^{-1}\oplus (S^T)^{-1}\oplus
S\oplus S\oplus S\oplus S.
\]
Hence, ${\cal T}(A,B)$ is congruent
to ${\cal T}(C,D)$.

Conversely, assume that ${\cal
T}(A,B)$ reduces to ${\cal T}(C,D)$
by congruence transformations and
substitutions \eqref{1.3z}; we need
to prove that $(A,B)$ is similar to
$(C,D)$. These transformations are
independent; we can first produce
all substitutions reducing
\[
(M_1,M_2,M_3(A,B)):={\cal T}(A,B)
\]
to
\begin{equation}\label{der}
(\gamma _{i1}M_1+\gamma _{i2}M_2
+\gamma _{i3}M_3(A,B))_{i=1}^3
\qquad\text{($[\gamma _{ij}]$ is
nonsingular)},
\end{equation}
and then all congruence
transformations and obtain
\begin{equation}\label{ijy}
(M_1,M_2,M_3(C,D))={\cal T}(C,D).
\end{equation}
Since \eqref{der} and \eqref{ijy}
are congruent,
\[
\rank{(\gamma _{i1}M_1+\gamma
_{i2}M_2 +\gamma _{i3}M_3(A,B))}=
\begin{cases}
\rank{M_i} &\text{if $i=1$ or
$i=2$,}\\ \rank{M_3(C,D)} &\text{if
$i=3$,}
\end{cases}
\]
and so $\gamma _{ij}=0$ if $i\ne j$
because of the form \eqref{5.9'} of
matrices of ${\cal T}(x,y)$; that
is, ${\cal T}(C,D)$ is congruent to
\begin{equation*}\label{lok}
(\gamma _{11}M_1,\, \gamma
_{22}M_2,\, \gamma _{33}M_3(A,B)).
\end{equation*}
Since $\mathbb F$ is algebraically
closed, the last triple is congruent
to
\[
\gamma _{11}^{-1/2}(\gamma
_{11}M_1,\, \gamma _{22}M_2,\,
\gamma _{33}M_3(A,B))\gamma
_{11}^{-1/2}.
\]
Hence,
\begin{equation*}\label{5a}
\text{$\ {\cal T}(C,D)\ $ is
congruent to $\ (M_1,\,\alpha M_2,\,
\beta M_3(A,B))$},
\end{equation*}
in which $\alpha :=\gamma
_{22}/\gamma _{11}$ and $\beta
:=\gamma _{33}/\gamma _{11}$.

By \eqref{5.9'},
$(I_1,\,I_1,\,I_1)^{\triangledown}$
is a direct summand of
$(M_1,\,M_2,\, M_3(A,B))$ for
congruence. Hence, $(I_1,\,\alpha
I_1,\,\beta I_1)^{\triangledown}$ is
a direct summand of $(M_1,\,\alpha
M_2,\, \beta M_3(A,B))$ for
congruence. Lemma \ref{lem2}(b)
ensures that each decomposition of
${\cal T}(C,D)$ by congruence
transformations into a direct sum of
indecomposable triples must have a
direct summand that is congruent to
$(I_1,\,\alpha I_1,\,\beta
I_1)^{\triangledown}.$

By simultaneous permutations of
rows and columns, ${\cal
T}(C,D)$ reduces to a direct sum
of triples of the form

\begin{equation}\label{vvv}
(I_1,0_1,0_1)^{\triangledown},\
\
(0_1,I_1,0_1)^{\triangledown},\
\
(0_1,0_1,I_1)^{\triangledown},\
\
(I_1,I_1,I_1)^{\triangledown},
\end{equation}
and of the triple ${\cal
G}(C,D)^{\triangledown}$ defined in
\eqref{5.q'}.

The triple ${\cal
G}(C,D)^{\triangledown}$ has no
direct summand $(I_1,\,\alpha
I_1,\,\beta I_1)^{\triangledown}$
for congruence since the pair
obtained from ${\cal
G}(C,D)^{\triangledown}$ by deleting
its last matrix is permutationally
congruent to
\[
(I_4,J_4(0))^{\triangledown}
\oplus\dots\oplus
(I_4,J_4(0))^{\triangledown} \quad
(\text{$n$ summands});
\]
this pair has no direct summand
$(I_1,\alpha I_1)^{\triangledown}$
for equivalence, and so for
congruence too. By Lemma
\ref{lem2}(b), $(I_1,\,\alpha
I_1,\,\beta I_1)^{\triangledown}$ is
congruent to one of the triples
\eqref{vvv}, hence it is congruent
to $(I_1,I_1,I_1)^{\triangledown}$,
and so $\alpha =\beta =1$; that is,
$ {\cal T}(A,B)\quad\text{is
congruent to}\quad {\cal T}(C,D). $
Due to \eqref{5.9'}, all the direct
summands of ${\cal T}(A,B)$ and
${\cal T}(C,D)$ coincide except for
${\cal G}(A,B)^{\triangledown}$ and
${\cal G}(C,D)^{\triangledown}$. By
Lemma \ref{lem2}(b), the triples
${\cal G}(A,B)^{\triangledown}$ and
${\cal G}(C,D)^{\triangledown}$ are
congruent. By Theorem
\ref{th2.1}(a), $(A,B)$ and $(C,D)$
are similar.
\end{proof}

\begin{corollary}\label{r6}
Let $U$ and $V$ be vector spaces and
$\dim V=3$. The problem of
classifying tensors $T\in U\otimes
U\otimes V$ that are symmetric
$($respectively, skew-symmetric$)$
on $U$ is wild since it reduces to
the classification problems
considered in Lemma \ref{t4.2}.
\end{corollary}

\section{Description of
Lie algebras with central
commutator subalgebra of
dimension $2$} \label{s5}

In this section, we describe
Lie algebras with central
commutator subalgebra of
dimension $2$ using the
canonical form of pairs of
skew-symmetric matrices for
congruence. An analogous
description of local
commutative algebras with
$(\Rad\Lambda)^3=0$ and $\dim
(\Rad\Lambda)^2=3$ would be
more awkward since the
classification of pairs of
symmetric matrices up to
congruence is more
complicated (see Thompson's
article \cite{thom} with an
extensive bibliography, or
\cite[Theorem 4]{ser1}).

The problem of classifying Lie
algebras with central commutator
subalgebra of dimension $1$ is
trivial: by Lemma \ref{l5z} each of
them is isomorphic to exactly one
algebra on ${\mathbb F}^{n}$ with
multiplication
\begin{equation*}
[u,v]:=\left(u^T\!\begin{bmatrix}
0_p&0&0\\0&0&I_q\\0&-I_q&0
\end{bmatrix}\!
v,\,0,\dots,0\right)^T,
\end{equation*}
given by natural numbers $p$ and $q$
such that $p+2q=n$.

Define the $(m-1)$-by-$m$
matrices
\begin{equation*}\label{4b}
F_m=\begin{bmatrix}
1&0&&0\\&\ddots&\ddots&\\0&&1&0
\end{bmatrix},\quad
G_m=\begin{bmatrix}
0&1&&0\\&\ddots&\ddots&\\0&&0&1
\end{bmatrix}
\end{equation*}
for each natural number $m$. In
particular, $F_1=G_1=0_{01}$ and so
$(F_1,G_1)^{\triangledown}=
(0_{1},0_{1})$ by \eqref{4w}.

\begin{theorem} \label{t5}
Let $\mathbb F$ be an
algebraically closed field of
characteristic other than
two. Let $L$ be a Lie algebra
over ${\mathbb F}$ whose
commutator subalgebra is
central and has dimension
$2$. Then $L$ is isomorphic
to an algebra on ${\mathbb
F}^{n}$ with multiplication
\begin{equation}\label{5.1}
[u,v]:=\left(u^T\!
\begin{bmatrix} 0_2&0\\0&A
\end{bmatrix}\!v,\,u^T\!
\begin{bmatrix}
0_2&0\\0&B
\end{bmatrix}\!v,\,
0,\dots,0 \right)^T,
\end{equation}
given by a pair $(A,B)$ of
skew-symmetric $(n-2)$-by-$(n-2)$
matrices of the form
\begin{equation}\label{5.2}
\bigoplus_{i=1}^p
(I_{l_i},J_{l_i}(\lambda_{i}))
^{\triangledown} \oplus
\bigoplus_{j=1}^q
(F_{r_j},G_{r_j})^{\triangledown},
\qquad p\ge 0,\ q\ge 0,
\end{equation}
$(J_l(\lambda)$ denotes the
$l$-by-$l$ Jordan block with
eigenvalue $\lambda$, and
$(\dots)^{\triangledown}$ is defined
in \eqref{4x'q} with
$\varepsilon=-1)$ except for the
case
\begin{equation}\label{wd}
\lambda_1=\dots=\lambda_p,\qquad
l_1=\dots=l_p=r_1=\dots=r_q=1.
\end{equation}
The sum \eqref{5.2} is determined by
$L$ uniquely up to permutations of
summands and up to linear-fractional
transformations of the sequence of
eigenvalues
\begin{equation}\label{4.1a}
(\lambda_1,\dots,\lambda_p)
\longmapsto \left(\frac{\gamma
+\delta \lambda_1} {\alpha +\beta
\lambda_1},\dots, \frac{\gamma
+\delta \lambda_p} {\alpha +\beta
\lambda_p}\right),
\end{equation}
in which all $\alpha +\beta
\lambda_i$ are nonzero and $\alpha
\delta -\beta \gamma \ne 0$.
\end{theorem}

\begin{proof}
By \cite{r.schar1}, \cite{ser1}, or
\cite{thom}, each pair of
skew-symmetric matrices over
$\mathbb F$ is congruent to a direct
sum of pairs of the form
\begin{equation}\label{eqq}
(I_m,
J_m(\lambda))^{\triangledown},\qquad
(J_m(0),I_m)^{\triangledown},\qquad
(F_m,G_m)^{\triangledown}
\end{equation}
(in the notation \eqref{4x'q} with
$\varepsilon=-1$), and this sum is
determined uniquely up to
permutation of summands.

Let $L$ be a Lie algebra of
dimension $n$ whose commutator
subalgebra is central and has
dimension $2$. By Lemma \ref{l5z}
for $t=2$, $L$ is isomorphic to an
algebra on ${\mathbb F}^{n}$ with
multiplication \eqref{5.1} given by
a pair $(A,B)$ of skew-symmetric
$(n-2)$-by-$(n-2)$ matrices, and
$(A,B)$ is determined by $L$
uniquely up to congruence and
invertible linear substitutions
\begin{equation}\label{5.n}
(A,B)\longmapsto (\alpha A+\beta
B,\,\gamma A+\delta B),\qquad \alpha
\delta -\beta \gamma \ne 0.
\end{equation}
By \eqref{eqq}, the pair $(A,B)$ is
congruent to a pair
\begin{equation}\label{4.2}
\bigoplus_{i=1}^{k}
(I_{l_i},J_{l_i}(\lambda_{i}))
^{\triangledown} \oplus
\bigoplus_{i=k+1}^{p}
(J_{l_i}(0),I_{l_i})
^{\triangledown} \oplus
\bigoplus_{j=1}^q(F_{r_j},G_{r_j})
^{\triangledown}
\end{equation}
determined by $(A,B)$ uniquely up to
permutation of summands.

Let us study how transformations
\eqref{5.n} change \eqref{4.2}. The
pairs \eqref{eqq} are indecomposable
with respect to congruence. Since
the first and the second pairs in
\eqref{eqq} have size $2m\times 2m$,
each indecomposable pair of
skew-symmetric matrices of size
$(2m+1)\times(2m+1)$ is congruent to
$(F_m,G_m)^{\triangledown}$. Each
transformation \eqref{5.n} is
invertible, so it transforms any
indecomposable pair of
skew-symmetric matrices to an
indecomposable one. Hence, although
each transformation \eqref{5.n} with
\eqref{4.2} may spoil summands
$(F_{r_j},G_{r_j})
^{\triangledown}$, but they are
restored by congruence
transformations.

If $k<p$, then we reduce the pair
\eqref{4.2} to a pair of the form
\eqref{5.2} (with other
$\lambda_1,\dots, \lambda_{k}$) as
follows. We convert all the summands
$(I_{l_i},J_{l_i}(\lambda_{i}))
^{\triangledown}$ and
$(J_{l_i}(0),I_{l_i})
^{\triangledown}$ to pairs with
nonsingular first matrices by any
transformation \eqref{5.n} given by
\[
\begin{bmatrix}\alpha &\beta \\
\gamma &\delta
\end{bmatrix}=
\begin{bmatrix}1&\beta \\0&1
\end{bmatrix},\quad \beta \ne 0,\
1+\beta \lambda_1\ne 0,\,\dots,\,
1+\beta \lambda_{k}\ne 0.
\]
Then we reduce each of these
summands to $(I,J_{l_i}(\lambda_i))
^{\triangledown}$ (with other
$\lambda_i$'s) by congruence
transformations using the following
fact: if matrix pairs $(M_1,M_2)$
and $(N_1,N_2)$ are equivalent,
i.e., $R(M_1,M_2)S=(N_1,N_2)$ for
some nonsingular $R$ and $S$, then
$(M_1,M_2)^{\triangledown}$ and
$(N_1,N_2)^{\triangledown}$ are
congruent:
\begin{equation}\label{ttt}
\begin{bmatrix}
  R&0\\ 0 &S^T
\end{bmatrix}
\begin{bmatrix}
  0&M_i\\ -M_i^T &0
\end{bmatrix}
\begin{bmatrix}
  R^T&0\\ 0 &S
\end{bmatrix}=
\begin{bmatrix}
  0&N_i\\ -N_i^T &0
\end{bmatrix}.
\end{equation}

Every transformation \eqref{5.n} for
which all $\alpha +\beta \lambda_i$
are nonzero, converts the summands
$(I_{l_i},J_{l_i}(\lambda_i))
^{\triangledown}$ of \eqref{5.2} to
the pairs \[\left(\alpha
I_{l_i}+\beta J_{l_i}(\lambda_i),
\gamma I_{l_i}+\delta
J_{l_i}(\lambda_i)\right)
^{\triangledown};\] by \eqref{ttt}
they are congruent to
\begin{equation}\label{4.4}
\left(I_{l_i},\ (\alpha I_{l_i}
+\beta J_{l_i}(\lambda_i))^{-1}
(\gamma I_{l_i} +\delta
J_{l_i}(\lambda_i))\right)
^{\triangledown}.
\end{equation}
The matrices $\alpha I_{l_i}+\beta
J_{l_i}(\lambda_i)$ and $\gamma
I_{l_i}+\delta J_{l_i}(\lambda_i)$
are triangular; their diagonal
entries are $\alpha +\beta
\lambda_i$ and $\gamma +\delta
\lambda_i$. Hence, the pair
\eqref{4.4} is congruent to
\[
\left(I_{l_i},\ J_{l_i}
\left(\frac{\gamma +\delta
\lambda_i} {\alpha +\beta \lambda_i}
\right)\right) ^{\triangledown},
\]
and the sequence of eigenvalues
changes by the rule
\eqref{4.1a}.

By Lemma \ref{l5z}, the matrices $A$
and $B$ in \eqref{5.1} must be
linearly independent. As follows
from \eqref{5.2}, $A$ and $B$ are
linearly dependent only if
\eqref{wd} holds.
\end{proof}

\begin{remark}
The theory of Lie rings and algebras
is tied to the theory of groups; see
\cite[Section 7]{bah} or \cite{zap}.
In particular, the results of
Sections \ref{s1} and \ref{s5} are
easily extended to every $p$-group
$G$ being the semidirect product of
the central commutator subgroup $G'$
of type $(p,\dots,p)$ and an abelian
group of type $(p,\dots,p)$. If $G$
is such a group, then
\[
G'=\langle
a_1\rangle_p\times\dots\times
\langle a_t\rangle_p ,\qquad
G/G'=\langle c_1\rangle_{p}
\times\dots\times \langle
c_n\rangle_{p}.
\]
Choosing $b_i\in c_i$, we may give
$G$ by the defining relations
\begin{gather*}\label{joj}
a_l^p=b_i^p=1,\qquad
[a_l,a_r]=[a_l,b_i]=1,\qquad
[b_i,b_j]=a_1^{\alpha_{1ij}}\cdots
a_t^{\alpha_{tij}},
\end{gather*}
in which $l,r\in\{1,\dots,t\}$,
$i,j\in\{1,\dots,n\}$, and
\[A_1=[\alpha_{1ij}],\,\dots,\,
A_t=[\alpha_{tij}]\] are linearly
independent skew-symmetric
$n$-by-$n$ matrices over the field
$\mathbb F_p$ of $p$ elements.
Conversely, each tuple
$(A_1,\dots,A_t)$ of linearly
independent skew-symmetric
$n$-by-$n$ matrices over $\mathbb
F_p$ gives such a group, and two
tuples give isomorphic groups if and
only if one reduces to the other by
congruence transformations and
substitutions \eqref{1.3z}, in which
the matrix $[\gamma_{ij}]$ is
nonsingular. Reasoning as in Theorem
\ref{t5}, we can describe such
groups having $G'$ of order $p^2$.
(A canonical form for congruence of
a pair of skew-symmetric matrices
over an arbitrary field is a direct
sum of pairs of the form \eqref{eqq}
with the Frobenius blocksa instead
of the Jordan blocks
$J_m(\lambda)$.) The problem of
classifying such groups with $G'$ of
order $p^3$ is hopeless since it
reduces to the problem of
classifying pairs of matrices over
$\mathbb F_p$ up to similarity. By
\cite{ser3}, the problem of
classifying finite $p$-groups with
central commutator subgroup of order
$p^2$ is hopeless in the same way
both for the groups in which $G'$ is
cyclic and for the groups in which
$G'$ is of type $(p,p)$. All finite
$p$-groups with central commutator
subgroup of order $p$ are easily
classified; see \cite{lie} and
\cite{serg}.
\end{remark}

\end{document}